\documentclass[a4paper,leqno, 12pt]{article}

\usepackage{tikz}
\usetikzlibrary{matrix,arrows, patterns, positioning, calc, intersections,shapes,decorations.pathmorphing,decorations.markings}

\usepackage{tikz-cd}

\usepackage{tikz-3dplot}

\usepackage{amsmath}
\usepackage[utf8]{inputenc}
\usepackage{microtype}
\usepackage{amscd}
\usepackage{amsmath,amssymb,amsfonts,accents,stmaryrd}
\usepackage{xstring}
\usepackage{xcolor}
\usepackage[mathscr,mathcal]{euscript}
\usepackage{xypic}
\xyoption{rotate}
\usepackage[all,v2,cmtip,2cell]{xy}
\UseAllTwocells
\entrymodifiers={+!!<0pt,\fontdimen22\textfont2>}
\usepackage{url}
\usepackage{latexsym}
\usepackage{amsthm}
\usepackage{multicol}
\usepackage{enumitem}
\usepackage[colorlinks=true, linkcolor={blue!50!black}, pdfhighlight=/O, ocgcolorlinks=true]{hyperref}
\usepackage{graphicx}
\usepackage{color}
\usepackage{mathtools}
\usepackage[colorinlistoftodos]{todonotes}
\usepackage[affil-it]{authblk}

\usepackage{adjustbox}

\usepackage{cleveref}

\newcommand{\highlight}[1]{%
  \tikz[baseline=(X.base)] \node[fill=yellow, anchor=base] (X) {$#1$};%
}

\allowdisplaybreaks

\usepackage[a4paper, total={5.5in,8.5in}]{geometry}

\numberwithin{equation}{section}
\theoremstyle{plain}
\newtheorem*{theorem*}{Theorem}
\newtheorem{theorem}[equation]{Theorem}
\newtheorem{lemma}[equation]{Lemma}
\newtheorem{proposition}[equation]{Proposition}

\theoremstyle{definition}

\newtheorem{definition}[equation]{Definition}

\newtheorem{remark}[equation]{Remark}

\setlist[enumerate]{label=(\arabic*), leftmargin=*}
\setenumerate{label=(\arabic*), leftmargin=*}
\setlist[itemize]{label=$\vcenter{\hbox{\footnotesize$\bullet$}}$, leftmargin=*}
\setitemize{label=$\vcenter{\hbox{\footnotesize$\bullet$}}$, leftmargin=*}

\newcommand{\cat}[1]{
\StrLen{#1}[\mystrlen]
\ifnum\mystrlen=1 \mathscr{#1}
\else \mathrm{#1}
\fi}

\newcommand\numberthis{\addtocounter{equation}{1}\tag{\theequation}}

\makeatletter
\newcommand{\mytag}[2]{%
  \text{#1}%
  \@bsphack
  \begingroup
    \@onelevel@sanitize\@currentlabelname
    \edef\@currentlabelname{%
      \expandafter\strip@period\@currentlabelname\relax.\relax\@@@%
    }%
    \protected@write\@auxout{}{%
      \string\newlabel{#2}{%
        {#1}%
        {\thepage}%
        {\@currentlabelname}%
        {\@currentHref}{}%
      }%
    }%
  \endgroup
  \@esphack
}
\makeatother

\title{Koszul complexes and derived intersections}

\author[]{
	Tristan Bozec\thanks{Univ Angers, CNRS, LAREMA - UMR 6093, SFR MATHSTIC, F-49000 Angers, France \\
											\href{mailto:tristan.bozec@univ-angers.fr}{\textup{tristan.bozec@univ-angers.fr}}},  
	Julien Grivaux\thanks{UMR7586 -- Institut de Math\'ematiques de Jussieu-Paris Rive Gauche \\
Sorbonne Universit\'e -- 4, Place Jussieu -- case 247 -- 75252 Paris Cedex 05 \\
\href{mailto:julien.grivaux@sorbonne-universite.fr}{\textup{julien.grivaux@sorbonne-universite.fr}}}
}
\date{}

\begin{document}

\maketitle

\begin{abstract}
The aim of this article is to provide a complementary understanding to some results of the second author using the machinery of Koszul complexes, and to explain how this approach can provide a new description of projective derived intersections.
\end{abstract}

\tableofcontents

\section{Introduction}

The main objective of the article \cite{grivaux-documenta} (whose results were obtained independently in \cite{arinkin-caldararu-habliczek}) was the following one : given an ambient complex manifold $Z$, along with two closed complex submanifolds $X,Y$ intersecting cleanly but not necessarily transversally along a smooth submanifold $T$, compute the derived tensor product \[
\mathcal O_X\overset{\mathbb L}{\otimes}_{\mathcal O_Z}\mathcal O_Y\] in the bounded derived category $\mathrm{D}(Z)$ of coherent sheaves on $Z$, and determine when this object is formal. One answer is that the formality of the derived intersection $j^*_{Y/Z}j_{X/Z*} \mathcal{O}_X$ in $\mathrm{D}(Y)$ is equivalent to the splitting of the excess conormal sequence
\[
0 \rightarrow \mathcal{N} \rightarrow (N^{\vee}_{X/Z})_{|T} \rightarrow N^{\vee}_{T/Y} \rightarrow 0.
\]
Besides, one of the most surprising things about this derived intersection $j^*_{Y/Z}j_{X/Z*} \mathcal{O}_X$ is that although it is locally isomorphic to $j_{T/Y*} \mathrm{Sym}(\mathcal{N}[1])$, it is not always the image of  an element in $\mathrm{D}(T)$ : a cohomological obstruction is given in \cite[Prop. 8.1]{grivaux-documenta}. Although it was not noticed at the time, the case where $T$ is of codimension one in $Y$ is of particular importance : then, the cohomological obstruction of \textit{loc. cit.} is equivalent to the splitting of the excess conormal exact sequence, which means that the formality of the derived intersection in $\mathrm{D}(Y)$ is in fact \textit{equivalent} to the fact that the derived intersection comes from $T$.
\par \medskip
When one cycle is given as the zero locus of a cosection of a holomorphic vector bundle vanishing transversally, we can solve the structural sheaf of the cycle by the corresponding Koszul complex of the cosection and restrict it to the other cycle. It gives again a Koszul complex, which is a derived critical locus in the sense of \cite{Gabriele_2020}. In this paper, we study in detail these kinds of derived critical loci. They are attached to cosections that we called \emph{weakly regular}. The corresponding derived critical loci are locally of the form $j_*\mathrm{Sym}(\mathcal{N}[1])$ for some excess bundle $\mathcal{N}$ living on the (underived) vanishing locus measuring the defect of transversality of the vanishing of the cosection.
\par \medskip
The motivation for this study is twofold. The first one is that Koszul complexes attached to weakly regular cosections are a toy model for derived intersections. This model shares a lot of common properties with derived intersections. For instance, theorem \ref{div} is the Koszul analog of \cite[thm 1.1 and prop. 8.1]{grivaux-documenta}. The divisor case is particularly enlightning, because it is extremely simple and provides concrete examples of elements of the derived categories which are locally pushforwards of elements of a smooth divisor, but not globally. The minimal example of such a complex is obtained by performing a blowup : if $U$ is a disk of $\mathbb{C}^2$ containing the origin and we blow up the point as below
\[
\xymatrix{ \mathbb{P}^1 \ar[r]^{j} \ar[d] & \mathrm{Bl}_{(0,0)} U \ar[d]^-{p} \\
\{(0,0)\} \ar[r]^{i} & U
}
\]
then $p^* i_* \mathcal{O}_{(0,0)}$ is locally isomorphic to $j_*(\mathcal{O}_{\mathbb{P}^1} \oplus \Omega^1_{\mathbb{P}^1}(1) [1])$, but it doesn't belong to $j_* \mathrm{D}(\mathbb{P}^1)$. The object is explicitly realized by pulling back the Koszul conplex on the trivial rank $2$ bundle on $\mathbb{C}^2$ given by the cosection $(f, g) \mapsto zf+wg$. These kind of push-pull calculations on blowups enter in the framework of excess intersection formulas as developed initially by Grothendieck in $K$-theory ({see} for instance \cite[lemma 19 c)]{Borel-Serre}), and then in a lot of cohomology theories like Chow groups. However, we see that no such formula can be expected in derived categories.
\par \medskip
The second motivation is that for derived intersections inside a proective space, the toy model is as strong as the model : all projective derived intersections are in fact given by Koszul complexes of weakly regular cosections. This follows by combining the diagonal trick (intersecting $X$ and $Y$ in $Z$ is the same as intersecting $X \times Y$ with the diagonal of $Z$) together with the fact that the diagonal of $\mathbb{P}^N$ is the zero locus of a regular cosection on $\Omega^1_{\mathbb{P}^N}(1) \boxtimes \mathcal{O}_{\mathbb{P}^N}(-1)$. The diagonal trick has already been investigated in the case of arbitrary self-intersections ({see} \cite[\S 9]{grivaux-documenta}) but it didn't provide much new information. However, combining it with the very rich structure of the diagonal of $\mathbb{P}^N$, it allows us to produce an explicit complex calculating projective derived self-intersections.
\par \medskip
The organization of the paper is as follows : in \Cref{sec2} we develop the main setting (Koszul complex for weakly regular cosections). The main result is \Cref{div}, which states that the Koszul complex of a weakly regular section is formal if and only if the excess exact sequence splits, and provides a cohomological obstruction to realize the Koszul complex as the direct image of an object leaving in the vanishing locus of the cosection. In \Cref{3}, we develop three constructions where Koszul constructions for weakly regular sections are useful. \Cref{3.1} deals with the divisor case. This case is very explicit. In \Cref{3.2}, we deal with blowups. This section can be seen as a natural continuation of \cite[\S 15]{Borel-Serre}. Lastly, in \Cref{3.3}, we prove that a clean derived intersection in a projective space can be expressed as a Koszul complex of a weakly regular cosection. The final \Cref{4} is dedicated to the self-intersection of a smooth subvariety of a projective space. The results \Cref{can1} are optional, but they shed light on some aspects of projective geometry related to derived intersections. In \Cref{proj}, we prove that given a smooth subvariety $X$ of $\mathbb{P}^N$, one of the two Beilinson spectral sequences attached to the structural sheaf of $X$ allows to produce a locally free resolution of $\mathcal{O}_X$ on $\mathbb{P}^N$. In the last section, we provide in \Cref{AS} an explicit complex computing the structural sheaf of the derived intersection of $X$ in $\mathbb{P}^N$. We present two proofs of this result, the first uses the aforementioned resolution of $\mathcal{O}_X$, and the second one used the realisation of this derived intersection as a Koszul complex.

\section{Koszul complexes and weakly regular cosections} \label{sec2}

\subsection{Derived critical loci}
The basic setting we fix is as follows :
\begin{enumerate}
\item[--] $Y$ a connected complex manifold;
\item[--] $\mathcal{E}$ a holomorphic vector bundle on $Y$ of rank $r$;
\item[--] $s \colon \mathcal{E} \rightarrow \mathcal{O}_Y$ a cosection of $\mathcal{E}$;
\item[--] $T$ the vanishing locus of $s$, endowed with its schematic structure;
\item[--] $j$ the closed embedding of $T$ in $Y$.
\end{enumerate}

\begin{definition}
The Koszul complex $\mathcal{K}(\mathcal{E}, s)$ attached to the pair $(\mathcal{E}, s)$ is the complex $\mathrm{Sym}\,(\mathcal{E}[1])$ endowed with the differentials $\wedge^p \mathcal{E} \rightarrow \wedge^{p-1} \mathcal{E}$ given by
\[
s(e_1 \wedge \ldots \wedge e_p) = \sum_{i=1}^p (-1)^{i-1} s(e_i) e_1 \wedge \ldots \wedge e_{i-1} \wedge e_{i+1} \wedge \ldots \wedge e_{p}.
\]
\end{definition}

The Koszul complex is naturally a sheaf of commutative differential graded algebras, and $\mathcal{H}^0(\mathcal{K}(\mathcal{E}, s) \simeq j_* \mathcal{O}_T$. This sheaf is the structural sheaf of the derived critical locus of the section $s^{\vee}$ of $\mathcal{E}^{\vee}$ ({see} \cite[Prop. 2.4]{Gabriele_2020}). It is well-known that the morphism $\mathcal{K}(\mathcal{E}, s) \to j_* \mathcal{O}_T$ is a quasi-isomorphism, meaning that the derived critical locus has no extra derived structure, if and only if the cosection is regular. One implication is clear, and the equivalence is given for instance by \cite[Lemma 15.30.7]{stacks-project}.
\par \medskip
The regularity of $s$ does not imply the smoothness of $T$. However, if $T$ is smooth (or lci), $N_{T/Y}^{\vee}$ is locally free, and $s$ is regular if and only if the restriction $s_{|T} \colon \mathcal{E}_{|T} \rightarrow N^{\vee}_{T/Y}$ is an isomorphism. 

\subsection{Weakly regular cosections}

We now define the central notion of the paper : weakly regular cosections. This is a local interpolation between a regular section and a zero section, with the additional input that the underived critical locus is smooth.

\begin{definition}\label{wregdef}
Consider a triple $(Y,\mathcal E,s)$ as before, and let $T$ be the zero locus of $s$. Assume that $T$ is smooth. We say that the cosection $s$ is weakly regular if the following condition holds : for any point $t$ of $T$, there exists a neighborhood $U$ of $t$ in $Y$ and a vector bundle $\widetilde{\mathcal{E}}$ on $U$ such that $s_{|U}$ factors as
\[
\mathcal{E}_{|U} \twoheadrightarrow \widetilde{\mathcal{E}} \xrightarrow{\widetilde{s}} {\mathcal{I}_T}_{| U}
\]
where $(\tilde{\mathcal E},\tilde s)$ is regular, and the first arrow is surjective. Equivalently, $\mathcal{E}$ is the direct sum of two vector bundles, the cosection being zero on the first one and regular on the second one.
\end{definition}
\begin{remark}\label{rem1}
\begin{enumerate}
\item[(i)] If $T$ has codimension $1$ in $Y$, $s$ is automatically weakly regular. Indeed, it suffices to take for $\widetilde{\mathcal{E}}$ the locally free sheaf $\mathcal{I}_T$.
\item[(ii)] Let $s$ be a regular section on $Z$ whose vanishing locus $Y$ is smooth, and let $X$ be a smooth submanifold of $Z$ intersecting $Y$ cleanly, that is the scheme-theoretic intersection of $X$ and $Y$ is smooth\footnote{This is equivalent to the notion of \textit{linear intersection} of \cite{grivaux-documenta}.}. Then $s_{|X}$ is a weakly regular cosection of $\mathcal{E}_{|X}$. The corresponding kernel $\mathcal{N}$ is an excess bundle fitting in the exact sequence
\[
0 \to \mathcal{N} \to \left(N^\vee_{Y/Z}\right)_{|T} \to N^{\vee}_{ T/X} \to 0
\]
where $T=X \cap Y$. This is an easy calculation using local coordinates. 
\end{enumerate}
\end{remark}

\begin{proposition} \label{loc}
Given $(Y, \mathcal{E}, s)$ where $s$ is weakly regular, let $T$ be the vanishing locus of $s$ and set $\mathcal{N}=\ker(\mathcal{E}_{|T} \rightarrow N^{\vee}_{T/Y})$. Then :
\begin{enumerate}
\item[\emph{(i)}] $\mathcal{K}(\mathcal{E}, s)_{|T} \simeq \mathrm{Sym}(\mathcal{E}[1])$.
\item[\emph{(ii)}] $\mathcal{K}(\mathcal{E}, s)$ is locally quasi-isomorphic on $Y$ to the complex $j_{*}\,\mathrm{Sym} \,(\mathcal{N}[1])$ with zero differentials.
\item[\emph{(iii)}] For any $p$, the induced morphism
\[
\mathcal{H}^{-p} (\mathcal{K}(\mathcal{E}, s))_{|T} \rightarrow \mathcal{H}^{-p} (\mathcal{K}(\mathcal{E}, s)_{|T}) \simeq \wedge^p \mathcal{E}_{|T}
\] 
is injective and its image is $\wedge^p \mathcal{N}$. In particular, $\mathcal{H}^{-p} (\mathcal{K}(\mathcal{E}, s))$ is canonically isomorphic to $j_{*}\wedge^p \mathcal{N}$.
\end{enumerate}
\end{proposition}

\begin{proof}
Point (i) is obvious since $s$ vanishes on $T$. For (ii), since the statement is local, we can write $s:\mathcal{E} \twoheadrightarrow \widetilde{\mathcal{E}} \xrightarrow{\widetilde{s}} \mathcal{I}_T$. Denote by $\widetilde{\mathcal{N}}$ the kernel of the morphism $\mathcal{E} \rightarrow \widetilde{\mathcal{E}}$. It is a locally free sheaf satisfying $\widetilde{\mathcal{N}}_{|T}=\mathcal{N}$. We can split locally $\mathcal{E}$ as $\widetilde{\mathcal{E}} \oplus \widetilde{\mathcal{N}}$, which gives an isomorphism
\[
(\mathcal{E}, s) \simeq (\widetilde{\mathcal{N}}, 0) \oplus (\widetilde{\mathcal{E}}, \widetilde s).
\]
Then, 
\[
\mathcal{K}(\mathcal{E}, s) \simeq \mathcal{K}(\widetilde{\mathcal{N}}, 0)\otimes_{\mathcal{O}_Y} \mathcal{K}(\widetilde{\mathcal{E}}, \widetilde s) .
\]
Since $\mathrm{Sym}\,(\widetilde{\mathcal{N}}[1])$ is a bounded complex of locally free sheaves and $\mathcal{K}(\widetilde{E}, \widetilde{s})$ is quasi-isomorphic to $ j_{*}\, \mathcal{O}_T$, we obtain that $\mathcal{K}(\mathcal{E}, s)$ is quasi-isomorphic to  
$\mathrm{Sym}\,(\widetilde{\mathcal{N}}[1])\otimes_{\mathcal{O}_Y}  j_{*}\, \mathcal{O}_T$ which is $j_{*}\, \mathrm{Sym}\,({\mathcal{N}[1]})$. 
\par \medskip
We can now prove (iii). We can argue locally. The morphism we want to consider is the first row of the diagram
\[\!\!\!\!\!\!\!\!\!
\small{
\xymatrix{
\mathcal{H}^{-p} (\mathcal{K}({\mathcal{E}}, s))_{|T} \ar[d]_-{\wr} \ar[r] & \mathcal{H}^{-p}(\mathcal{K}(\mathcal{E}, s)_{|T}) \ar[d]^-{\wr} \ar[r]^-{\sim} & \wedge^p \mathcal{E}_{|T} \ar[d]^-{\wr} \\
\mathcal{H}^{-p} (\mathcal{K}(\widetilde{\mathcal{N}}, 0) \otimes_{\mathcal{O}_Y} \mathcal{K}(\mathcal{\widetilde{E}}, s))_{|T} \ar[r] \ar[d]_-{\wr}& \mathcal{H}^{-p} \left(\mathcal{K}(\widetilde{\mathcal{N}}, 0)_{|T} \otimes_{\mathcal{O}_T} \mathcal{K}(\widetilde{\mathcal{E}}, s)_{|T}\right) \ar[d] \ar[r]^-{\sim}& \bigoplus_{r+s=P} \wedge^r \widetilde{\mathcal{N}}_{|T} \otimes_{\mathcal{O}_T} \wedge^s \widetilde{\mathcal{E}}_{|T} \ar[d] \\
\mathcal{H}^{-p}(\mathcal{K}(\widetilde{\mathcal{N}},0) \otimes_{\mathcal{O}_Y} j_*\mathcal{O}_T)_{|T} \ar[r]^-{\sim} & \mathcal{H}^{-p} (\mathcal{K}(\widetilde{\mathcal{N}}, 0)_{|T}) \ar@<2ex>[u] \ar[r]^-{\sim}&  \ar@<2ex>[u] \wedge^p  \widetilde{\mathcal{N}}_{|T}
}
}
\]
The bottom right diagram commute, when considering downward, but also upward arrows. This gives the result.
\end{proof}

\subsection{Formality}
Let us consider a triple $(Y,\mathcal{E}, s)$ where $s$ is weakly regular with vanishing locus $T$. It yields an exact sequence
\begin{equation} \label{yolo3}
0 \to {\mathcal{N}} \to \mathcal{E} _{|T} \to N^{\vee}_{T/Y} \to 0. 
\end{equation}
on the smooth manifold $T$. 
Let us consider the four following assertions : 
\begin{enumerate}
\item[(i)] If $d=\mathrm{codim}_Y(T)$, then the map $\wedge^d \mathcal{E}_{|T} \rightarrow \mathrm{det}\, N^{\vee}_{T/Y}$ splits.
\item[(ii)] $\mathcal{K}(\mathcal{E}, s)\to j_*\mathcal O_T$ splits.
\item[(iii)] $\mathcal{K}(\mathcal{E}, s)$ belongs to $j_{*}(\mathrm{D}(T))$.
\item[(iv)] $\mathcal{K}(\mathcal{E}, s) \simeq j_{*}\, \mathrm{Sym}(\mathcal{N}[1])$ in $\mathrm{D}(Y)$.
\item[(v)] The exact sequence \eqref{yolo3} splits. 
\end{enumerate}

\begin{theorem} \label{div}
The complex $\mathcal{K}(\mathcal{E}, s)$ is isomorphic to $j_{*}\, \mathrm{Sym}(\mathcal{N}[1])$ in $\mathrm{D}(Y)$
 if and only if the exact sequence \eqref{yolo3} splits (that is properties \emph{(ii)}--\emph{(iv)} are equivalent). Besides, there is a diagram of implications:
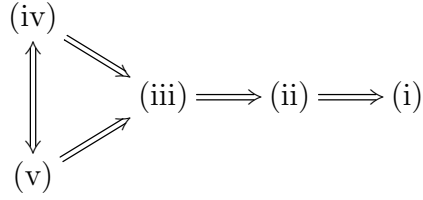
\[
\xymatrix @R=1em {\emph{(iv)} \ar@{<=>}[dd] \ar@{=>}[rd] &&& \\
& \emph{(iii)} \ar@{=>}[r]& \emph{(ii)}\ar@{=>}[r]&\emph{(i)}\\
\emph{(v)} \ar@{=>}[ru] &&&}
\]
and if $T$ has codimension one, all four properties \emph{(i)}--\emph{(iv)} are equivalent.
\end{theorem}

\begin{proof}
The obvious implications are $\mathrm{(iii)} \Rightarrow \mathrm{(ii)}$ and $\mathrm{(iv)} \Rightarrow \mathrm{(ii)}$. 
\par \medskip
\fbox{$\mathrm{(iv)} \Rightarrow \mathrm{(v)}$}
For any morphism
$\mathcal{A} \rightarrow j_* \mathcal{B}$,
where $\mathcal{A}$ is a bounded complex of locally free sheaves on $Y$ and $\mathcal{B}$ is in $\mathrm{D}(T)$, we get by adjunction a morphism $\mathcal{A}_{|T} \rightarrow \mathcal{B}$ and for any $k$, the diagram
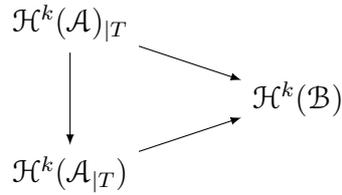
\[
\begin{tikzpicture}[>=latex]
	\node (A) at (0,2) {$\mathcal{H}^k(\mathcal{A})_{|T}$};        
	\node (B) at (0,0) {$\mathcal{H}^k(\mathcal{A}_{|T})$};        
	\node (C) at (3,1) {$\mathcal{H}^k(\mathcal{B})$};        
	
	\draw[->] (A) -- (B);         
	\draw[->] (A) -- (C);         
	\draw[->] (B) -- (C);         
\end{tikzpicture}
\]
commutes. Applying this to $\mathcal{A}=\mathcal{K}(\mathcal{E},s)$, $\mathcal{B}=\mathrm{Sym}\, (\mathcal{N}[1])$ and $k=-1$ together with \Cref{loc}(iii), we get a diagram
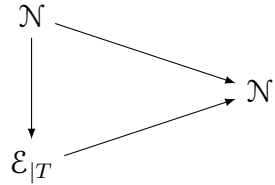
\[
\begin{tikzpicture}[>=latex]
	\node (A) at (0,2) {$\mathcal{N} $};        
	\node (B) at (0,0) {$\mathcal{E}_{|T}$};        
	\node (C) at (3,1) {$\mathcal{N}$};        
	
	\draw[->] (A) -- (B);         
	\draw[->] (A) -- (C);         
	\draw[->] (B) -- (C);         
\end{tikzpicture}
\]
which means that \eqref{yolo3} splits.
\par \medskip
\fbox{$\mathrm{(v)} \Rightarrow \mathrm{(iv)}$} There is a natural morphism from $\mathcal{K}(\mathcal{E}, s)$ to $j_* \mathrm{Sym}(\mathcal{N}[1])$ which is obtained by adjunction from the morphisms $\wedge^p \mathcal{E}_{|T} \to \wedge^p \mathcal{N}$ given by wedging $p$ times the retraction.
\par \medskip
\fbox{$\mathrm{(iii)}  \Rightarrow \mathrm{(ii)}$} Assume that $\mathcal{K}(\mathcal{E}, s)=j_{*} \theta$ for $\theta$ in $\mathrm{D}(T)$. There is a canonical morphism $\mathcal{O}_Y \rightarrow \mathcal K(\mathcal{E}, s)$ inducing the natural morphism $\mathcal{O}_Y \rightarrow \mathcal{O}_T$ in local cohomology of degree $0$. By adjunction, we have
\[
\mathrm{Hom}_{\mathrm{D}(Y)}(\mathcal{O}_Y,j_{*} \theta) \simeq \mathrm{Hom}_{\mathrm{D}(T)}(\mathcal{O}_T,\theta)
\]
so we get a morphism $\mathcal{O}_T \rightarrow \mathcal{\theta}$ inducing an isomorphism in local cohomology of degree zero and after applying $j_*$ we're done.
\par\medskip
\fbox{$\mathrm{(ii)}  \Rightarrow \mathrm{(i)}$} We follow the strategy of the proof of \cite[Prop. 8.1]{grivaux-documenta}. We have two arrows
\begin{align*}
\mathrm{Hom}_{\mathrm{D}(Y)}\left(j_* \mathcal{O}_T, \mathcal K(\mathcal{E}, s)\right) & \simeq \mathrm{Hom}_{\mathrm{D}(Y)}\left(\mathcal{O}_T, j^!\mathcal K(\mathcal{E}, s)\right) \\
& \simeq \mathbb{H}^0\left(j^*\mathcal K(\mathcal{E}, s) \overset{\mathbb{L}}{\otimes} \mathrm{det} \,N_{T/Y}[-d]\right) \\
& \rightarrow \mathbb{H}^{-d}\left(\mathrm{Sym}\,(\mathcal{E}_{|T}[1]) \otimes \mathrm{det} \,N_{T/Y}\right) \\
& \simeq \bigoplus_{i \geq d} \mathrm{H}^{i-d}\left(\wedge^i \mathcal{E}_{|T} \otimes \mathrm{det} \, N_{T/Y}\right) \\
& \rightarrow \mathrm{H}^0\left(\wedge^d \mathcal{E}_{|T} \otimes \mathrm{det} \, N_{T/Y}\right)
\end{align*}
and
\begin{align*}
\mathrm{Hom}_{\mathrm{D}(Y)}\left(j_* \mathcal{O}_T, j_* \mathcal{O}_T\right) & \simeq \mathrm{Hom}_{\mathrm{D}(Y)}\left(\mathcal{O}_T, j^! j_* \mathcal{O}_T\right) \\
& \simeq \mathrm{H}^0\left(j^* j_* \mathcal{O}_T \overset{\mathbb{L}}{\otimes} \mathrm{det} \,N_{T/Y}[-d]\right) \\
& = \mathrm{H}^d\left(j^* j_* \mathcal{O}_T \overset{\mathbb{L}}{\otimes} \mathrm{det} \,N_{T/Y}\right) \\
& \rightarrow  \mathrm{H}^0\left(\mathcal{H}^d\Big(j^* j_* \mathcal{O}_T \overset{\mathbb{L}}{\otimes} \mathrm{det} \,N_{T/Y}\Big)\right) \\
&= \mathrm{H}^0\left(\mathcal{T}or^d_{\mathcal{O}_Y}(\mathcal{O}_T, \mathcal{O}_T) \otimes \mathrm{det} \,N_{T/Y}\right) \\
&= \mathrm{H}^0 \left(\mathcal{O}_Y\right).
\end{align*}
A local computation yields that the diagram 
\[
\xymatrix{
\mathrm{Hom}_{\mathrm{D}(Y)}(j_* \mathcal{O}_T, \mathcal K(\mathcal{E}, s)) \ar[r] \ar[d] & \mathrm{H}^0(\wedge^d \mathcal{E}_{|T} \otimes \mathrm{det} \, N_{T/Y}) \ar[d] \\
\mathrm{Hom}_{\mathrm{D}(Y)}(j_* \mathcal{O}_T, j_* \mathcal{O}_T) \ar[r] & \mathrm{H}^0 (\mathcal{O}_Y)
}
\]
commutes, where the right map is induced by $\wedge^d \mathcal{E}_{|T} \rightarrow \wedge^d N^{\vee}_{T/Y}$. This gives the result. 
\par \medskip
If $T$ has codimension one, then $\mathrm{(i)} \Leftrightarrow \mathrm{(v)}$ and all propreties are equivalent. 
\end{proof}

\section{Examples} \label{3}

\subsection{Divisor case} \label{3.1}
The case where $T$ has codimension one is of particular interest, because it allows to describe explicitly 
$\mathcal{K}(\mathcal{E}, s)$ with the help of the exact sequence
\begin{equation} \label{yolo}
0 \rightarrow \widetilde{\mathcal{N}} \rightarrow \mathcal{E} \rightarrow \mathcal{I}_T \rightarrow 0.
\end{equation}
Indeed, for any integer $p$, we have an exact sequence 
\begin{equation} \label{simple}
0 \to \wedge^p \widetilde{\mathcal{N}} \to \wedge^p \mathcal{E} \to \wedge^{p-1} \widetilde{\mathcal{N}} \otimes \mathcal{I}_T \to 0.
\end{equation}
The differential of the Koszul complex is given by the composition
\[
\wedge^p \mathcal{E} \to \wedge^{p-1} \widetilde{\mathcal{N}} \otimes_{\mathcal{O}_Y} \mathcal{I}_T \hookrightarrow \wedge^{p-1} \widetilde{\mathcal{N}} \to \wedge^{p-1} \mathcal{E}
\]
using \eqref{simple} and the injection $\mathcal{I}_T \hookrightarrow \mathcal{O}_Y$. Using this description, we see immediately that
\[
\mathcal{H}^{-p} (\mathcal{K}(\mathcal{E}, s)) \simeq \wedge^p \widetilde{\mathcal{N}} /  \wedge^p \widetilde{\mathcal{N}} \otimes_{\mathcal{O}_Y} \mathcal{I}_T \simeq j_{*}\wedge^p \mathcal{N}
\]
as predicted by \Cref{loc}(i). 
\par \medskip
We can provide a simpler proof of the implication $\mathrm{(iii)} \Rightarrow \mathrm{(iv)}$ of \Cref{div} in this particular case. There is a morphism $(\widetilde{\mathcal{N}},0) \rightarrow (\mathcal{E}, s)$, hence a morphism
$\phi \colon \mathrm{Sym}\,(\widetilde{\mathcal{N}}[1]) \rightarrow \mathcal{K}(\mathcal{E}, s)$, which induces on local cohomology sheaves the natural projection
\[
\mathcal{H}^{-p}(\phi) \colon \wedge^p \widetilde{\mathcal{N}} \to \mathcal{H}^{-p}(\mathcal{K}(\mathcal{E}, s)) \simeq 	j_{*} \wedge^p \mathcal{N}
\]
via \Cref{loc}(ii).
By adjunction, we get a morphism
\begin{align*}
\mathrm{Hom}_{\mathrm{D}^{-}(Y)}\big(\mathrm{Sym}\,(\widetilde{\mathcal{N}}[1]),j_{*} \theta\big) &\simeq \mathrm{Hom}_{\mathrm{D}^{-}(T)}\big(j_{T/Y}^* \mathrm{Sym}\,(\widetilde{\mathcal{N}}[1]),\theta\big) \\
&= \mathrm{Hom}_{\mathrm{D}^{-}(T)}\big(\mathrm{Sym}\,({\mathcal{N}}[1]),\theta\big) \\
& \rightarrow \mathrm{Hom}_{\mathrm{D}^{-}(Y)}\big(j_{*} \mathrm{Sym}\,({\mathcal{N}}[1]),j_{*}\theta\big) \\
&=  \mathrm{Hom}_{\mathrm{D}^{-}(Y)}\big(j_{*} \mathrm{Sym}\,({\mathcal{N}}[1]),\mathcal{K}(\mathcal{E}, s)\big) 
\end{align*}
By \Cref{loc}(ii), the image of $\phi$ is an isomorphism on cohomology sheaves, so it is an isomorphism in the derived category $\mathrm{D}(Y)$.
\par \medskip
We end this section giving an explicit description of the length two truncations of $\mathcal{K}(\mathcal{E},s)$. Let $\overline{T}$ be the first formal neighborhood of $T$ in $Y$. Let $k \colon \overline{T} \hookrightarrow Y$ be the closed immersion of $\overline{T}$ in $Y$. 

\begin{definition}
The element $\delta$ in $\mathrm{Hom}_{\mathrm{D}(Y)}(j_* \mathcal{O}_T, j_* \mathcal N[2])$ is the composition of the extension class of the exact sequence
\begin{equation} \label{ati}
0 \to j_* N^{\vee}_{T/Y} \to k_* \mathcal{O}_{\overline{T}} \to j_* \mathcal{O}_T \to 0.
\end{equation}
 with the extension class of the restriction of \eqref{yolo} to $T$.
\end{definition}

\begin{proposition} \label{2}
For any nonnegative integer $p$, the morphism from $j_* \wedge^p \mathcal{N}$ to $j_* \wedge^{p+1} \mathcal{N} [2]$ given by the truncation distinguished triangle
\[
j_* \wedge^{p+1} \mathcal{N} \,[p+1] \to \tau^{[-(p+1), -p]} \, \mathcal{K}(\mathcal{E}, s) \to j_* \wedge^{p} \mathcal{N} \,[p] \xrightarrow{+1}
\]
is given by the composition
\[
j_* \wedge^{p} \mathcal{N} \,[p] \xrightarrow{\wedge^{p} \widetilde{\mathcal{N}}[p] \overset{\mathbb{L}}{\otimes} \delta} \wedge^{p} \widetilde{\mathcal{N}} \,[p]  \otimes j_* \mathcal{N}[2] \simeq j_* \left(\wedge^{p} \mathcal{N} \,[p] \otimes \mathcal{N}[2] \right) \xrightarrow{\, \wedge \, } j_*  \wedge^{p+1} \mathcal{N} \,[p+2].
\] 
\end{proposition}

\begin{proof}
The morphism $\delta$ can be written as the composition of the extension class of 
\[
0 \to j_* \mathcal{I}_T \to \mathcal{O}_Y \to j_* \mathcal{O}_T \to 0
\]
together with \eqref{yolo} and then the restriction to $T$. Hence $\mathrm{cone}\,(\delta)[-1]$ is isomorphic to 
\[
\widetilde{\mathcal{N}} \otimes \mathcal{I}_T \to \mathcal{E}\to \mathcal{O}_Y
\]
where $\mathcal{O}_Y$ sits in degree $0$. After tensorizing with $\wedge^{p} \widetilde{\mathcal{N}} [p] $ we get the complex.
\[
\wedge^p \widetilde{\mathcal{N}} \otimes \widetilde{\mathcal{N}} \otimes \mathcal{I}_T \to \wedge^p \widetilde{\mathcal{N}} \otimes \mathcal{E} \to  \wedge^p \widetilde{\mathcal{N}}
\]
Then we can consider the morphism
\[
\xymatrix{
\wedge^p \widetilde{\mathcal{N}} \otimes \widetilde{\mathcal{N}} \otimes \mathcal{I}_T \ar[r] \ar[d] &  \wedge^p \widetilde{\mathcal{N}} \otimes \mathcal{E} \ar[r] \ar[d]  & \wedge^p  \widetilde{\mathcal{N}} \ar@{=}[d] \\
\wedge^{p+1} \widetilde{\mathcal{N}} \otimes \mathcal{I}_T \ar[r] &\wedge^{p+1} \mathcal{E} \ar[r] & \wedge^{p} \widetilde{\mathcal{N}} 
}
\]
which corresponds exactly to the last step of the symmetrization process. The bottom line is exactly the two step troncature of $\mathcal{K}(s, \mathcal{E})$, whence the result. 
\end{proof}
\subsection{Blowups} \label{3.2}

Let $\hat Y=\mathrm{Bl}_TY$ be the blow-up of $Y$ along $T$, and let 
\[
\xymatrix{E\ar[r]^-k\ar[d]_-q&\hat Y\ar[d]^-p \\
T\ar[r]^-j&Y}
\] 
be the corresponding diagram, where $E=\mathbb{P}(N_{T/Y})$ is the  exceptional divisor. The conormal excess bundle $F$ of the blowup is defined by the short exact sequence
\begin{equation} \label{excess}
0 \rightarrow F \rightarrow q^* N^{\vee}_{T/Y} \rightarrow N^{\vee}_{E/\widehat{Y}} \rightarrow 0.
\end{equation}
Since $N^{\vee}_{E/\widehat{Y}} =\mathcal{O}_{\mathbb{P}(N_{T/Y})}(1)$, 
\eqref{excess} is the relative Euler exact sequence \eqref{euler} for the projective bundle $E=\mathbb{P}(N_{T/X})$, and $F$ is the twisted relative cotangent bundle $\Omega^1_{E/T}(1)$. 
In this section, we study the complex $p^{\star} j_{\star} \mathcal{O}_T$. It has been known for a while that 
\[
\mathcal{H}^{-k}(p^* j_* \mathcal{O}_T) \simeq \wedge^k F
\] 
(see \cite[\S 15]{Borel-Serre}, where  Grothendieck calls $F^*$ what we call $F$). Then we obtain the following.

\begin{proposition}
The complex $p^{*} j_{*} \mathcal{O}_T$ in $\mathrm{D}(\widehat{Y})$ does not belong to $k_* \mathrm{D}(E)$ if $T$ has codimension at least $2$.
\end{proposition}

We present two proof of this result, one relying on Koszul complexes, the second one on derived intersections.

\begin{proof}[Proof 1]
Since the statement is local on $Y$, we can assume that $T$ is the vanishing locus of a regular cosection of a vector bundle $\mathcal{E}$ on $Y$ such that $\mathcal{E}_{|T}=N^{\vee}_{T/Y}$. Then $p^*j_{*} \mathcal{O}_T$ is isomorphic to $\mathcal K(p^{*} \mathcal{E}, p^{*} s)$. Besides, the exact sequence \eqref{yolo3} attached to the couple $(p^*\mathcal{E}, p^*s)$ is exactly \eqref{excess}. It is never split on any fiber of $\mathbb{P}(N_{T/Y})$, so it is never split on $E$. Then we apply \Cref{div}.
\end{proof} 

\begin{proof}[Proof 2]
We can decompose the projection $p$ throught the graph construction and the projection
\[
\widehat{Y} \xrightarrow{(p, \mathrm{id})} Y \times \widehat{Y} \xrightarrow{\mathrm{pr}_1} Y
\]
so that
\[
p^* j_*  \mathcal{O}_T =(p, \mathrm{id})^* \,j_{T \times \widehat{Y} / Y \times \widehat{Y}*} \mathcal{O}_{T \times \widehat{Y}}.
\]
Hence $p^*j_* \mathcal{O}_T$ is the derived intersection of $T \times Y$ and $\widehat{Y}$ in $Y \times \widehat{Y}$, considered in the derived category of $\widehat{Y}$. Besides, the excess conormal sequence attached to this intersection is the exact sequence \eqref{excess} which is never split. Thanks to \cite[Prop. 8.1]{grivaux-documenta}, $p^*j_* \mathcal{O}_T$ does not belong to $k_* (\mathrm{D}(E))$.
\end{proof}

\subsection{Derived projective intersections} \label{3.3}
Let $W$ be a finite dimensional complex vector space. On $\mathbb{P}(W)$, we have the Euler exact sequence
\begin{equation} \label{euler}
0 \rightarrow \Omega^1(1) \rightarrow W^* \otimes_{\mathbb{C}} \mathcal{O}\rightarrow \mathcal{O}(1) \rightarrow 0
\end{equation}
where we drop everywhere the subscript $\mathbb{P}(W)$ to lighten notation. The vector bundle $ \mathcal{E}=\Omega^1 (1) \boxtimes \mathcal{O}(-1)$ on $\mathbb{P}(W) \times\mathbb{P}(W) $ is endowed with a natural cosection $s$ given by the morphism
\begin{align*}
\Omega^1 (1) \boxtimes \mathcal{O}(-1) &\rightarrow (W^* \otimes_{\mathbb{C}} \mathcal{O}) \boxtimes (W \otimes_{\mathbb{C}} \mathcal{O}) \\
&=(W^* \otimes W) \otimes_{\mathbb{C}} \mathcal{O}_{\mathbb{P}(W) \times \mathbb{P}(W)} \\
& \rightarrow \mathcal{O}_{\mathbb{P}(W) \times \mathbb{P}(W)}.
\end{align*}
It is well-known and easy to check that this cosection is regular, and its vanishing locus is the diagonal of $\mathbb{P}(W)$ ({see} \cite{Beilinson}, \cite{Orlov}).

\begin{proposition} \label{brique}
Let $U$ be a Zariski open subset of $\mathbb{P}(W)$. Let $X$ and $Y$ be closed subschemes of $U$ intersecting cleanly along $T$, and let $\delta$ be the diagonal injection of $U$. Then in $\mathrm{D}(U^2)$ we have \[
\delta_{*} (\mathcal{O}_X \overset{\mathbb{L}}{\otimes}_{\mathcal{O}_{U}} \mathcal{O}_Y) \simeq j_{X \times Y / U^2*}\mathcal{K}(\mathcal{E}_{| X \times Y}, s_{ | X \times Y}).\]. 
\end{proposition}

\begin{remark}
The vanishing locus of the cosection $s_{|X \times Y}$ is $\Delta_T$. Thanks to~\Cref{rem1}(ii), since $X\times Y$ and $\Delta_{U}$ intersect cleanly, this cosection is weakly regular.
\end{remark}

\begin{proof}
We have
\[ \mathcal{O}_X \overset{\mathbb{L}}{\otimes}_{\mathcal{O}_{U}} \mathcal{O}_Y \simeq \delta^* (\mathcal{O}_X \boxtimes \mathcal{O}_Y)=\delta^*\, \mathcal{O}_{X\times Y}
\]
whence
\begin{align*}
\delta_*(\mathcal{O}_X \overset{\mathbb{L}}{\otimes}_{\mathcal{O}_{U}} \mathcal{O}_Y) &\simeq \mathcal{O}_{X \times Y} \overset{\mathbb{L}}{\otimes}_{\mathcal{O}_{U^2}} \mathcal{O}_{\Delta_{U}} \\
 &\simeq \mathcal{O}_{X \times Y} \overset{\mathbb{L}}{\otimes}_{\mathcal{O}_{U^2}} \mathcal{K}(\mathcal{E}, s) \\
& \simeq \mathcal{O}_{X \times Y} {\otimes}_{\mathcal{O}_{U^2}} \mathcal{K}(\mathcal{E}, s) \\
& \simeq j_{X \times Y / U^2*} \,\mathcal{K}(\mathcal{E}_{| X\times Y}, s_{|X\times Y}).\qedhere
\end{align*}
\end{proof}

\section{Projective self-intersections} \label{4}
We place ourselves in the setting of \Cref{proj}, with $X=Y$ smooth of dimension $d$ in $\mathbb{P}(W)$, where $\dim W = N+1$, and $U=\mathbb{P}(W)$. Let $c=N-d$ be the codimension of $X$ in $\mathbb{P}(W)$. 

\subsection{The canonical complex (I)} \label{can1}
We recall the following well-known fact (it is the easy case of Fujita's conjecture).
\begin{lemma}  \label{fujita}
The line bundle $K_X(d+1)$ is globally generated.
\end{lemma}

\begin{proof}
We argue by induction on $d$. For any point $x$ in $X$, let $H$ be a smooth hyperplane section of $X$ passing through $x$, and $j \colon H \rightarrow X$ the injection. Then we have
\[
0 \rightarrow K_X(d) \rightarrow K_X(d+1) \rightarrow j_* K_H (d) \rightarrow 0.
\]
If $d \geq 1$, $H^1(X,K_X(d))=0$ thanks to Kodaira--Nakano vanishing theorem, so the map $H^0(X, K_X(d+1)) \rightarrow H^0(H, K_H(d))$ is onto. Since $K_H(d)$ is globally generated by induction, there exists a section of $K_H(d)$ which is nonzero at $x$, and the result follows.
\end{proof}

\begin{definition}
We define the \emph{canonical complex} $(\mathscr{C}(X), \delta)$ of $X$ by
\[
\mathscr{C}_{-k}(X) =\Omega^{k}_{|X}(k)\otimes_\mathbb C H^0\big(X,K_X(k)\big)^*
\]
and the differential $\delta_{-k}$ is given by the composition
\[
\xymatrix{
\Omega^{k}_{|X}(k)\otimes_\mathbb C H^0\big(X, K_X(k)\big)^* \ar[d] \\
\Omega^{k-1}_{|X}(k-1)\otimes_{\mathcal{O}_X} \Omega^{1}_{|X}(1) \otimes_\mathbb C H^0\big(X,K_X(k)\big)^* \ar[d] \\ 
\Omega^{k-1}_{|X}(k-1)\otimes_{\mathbb{C}} W^* \otimes_\mathbb C H^0\big(X,K_X(k)\big)^* \ar[d] \\
\Omega^{k-1}_{|X}(k-1)\otimes_{\mathbb{C}} H^0(X, \mathcal{O}(1))\otimes_\mathbb C H^0\big(X,K_X(k)\big)^* \ar[d] \\
\Omega^{k-1}_{|X}(k-1)\otimes_\mathbb C H^0\big(X, K_X(k-1)\big)^*.
}
\]
\end{definition}
Since $K_X(N)$ is globally generated, we can see $(K_X(N))^*$ as a subbundle of $H^0(X, K_X(N))^* \otimes_{ \mathbb{C}} \mathcal{O}_X$. Hence 
\[
\wedge^d N^*_{X/\mathbb{P}(W)} \simeq \Omega^N_{|X} \otimes  K_X^* \simeq \Omega^N_{|X}(N) \otimes (K_X(N))^*
\]
is naturally a subbundle of $\mathscr{C}_{-N}(X)$.

\begin{lemma} \label{knife}
The line bundle $\wedge^d N^*_{X/\mathbb{P}(W)}$ lies in $\ker \delta_{-N}$.
\end{lemma}

\begin{proof}
This follows from the diagram \footnotesize
\[\!\!\!\!\!\!\!\!\!\!
\begin{tikzcd}
	\Omega^{k}_{|X}(N)\underset{\mathcal{O}_X}{\otimes}  K_X(N)^* \arrow[r] \arrow[d] & \Omega^{k}_{|X}(N)\underset{\mathbb C}{\otimes}  H^0\left(X,K_X(N)\right)^* \arrow[d] \\
	\Omega^{k}_{|X}(N)\underset{\mathcal O_X}{\otimes}\Omega^{1}_{|X}(1)\underset{\mathcal{O}_X}{\otimes}  K_X(N)^* \arrow[r] \arrow[d]\arrow[ddd, bend right=85,"0"']  & \Omega^{k}_{|X}(N)\underset{\mathcal O_X}{\otimes}\Omega^{1}_{|X}(1)\underset{\mathbb C}{\otimes}  H^0\left(X,K_X(N)\right)^* \arrow[d] \\
	\Omega^{k}_{|X}(N)\underset{\mathbb C}{\otimes}W^*\underset{\mathcal{O}_X}{\otimes}  K_X(N)^* \arrow[r] \arrow[d] & \Omega^{k}_{|X}(N)\underset{\mathbb C}{\otimes}W^*\underset{\mathbb C}{\otimes}  H^0\left(X,K_X(N)\right)^*  \arrow[d] \\
	\Omega^{k}_{|X}(N)\underset{\mathbb C}{\otimes}H^0(X, \mathcal{O}(1))\underset{\mathcal{O}_X}{\otimes}  K_X(N)^* \arrow[r] \arrow[d] & \Omega^{k}_{|X}(N)\underset{\mathbb C}{\otimes}H^0(X, \mathcal{O}(1))\underset{\mathbb C}{\otimes}  H^0\left(X,K_X(N)\right)^* \arrow[d] \\
	\Omega^{k}_{|X}(N-1)\underset{\mathcal{O}_X}{\otimes}  K_X(N-1)^* \arrow[r]           & \Omega^{k}_{|X}(N-1)\underset{\mathbb C}{\otimes}  H^0\left(X,K_X(N-1)\right)^*.
\end{tikzcd}
\]
\end{proof}

\subsection{The Beilinson spectral sequence} \label{proj}
Recall that for any coherent sheaf $\mathcal{F}$ on $\mathbb{P}(W)$, we have isomorphisms
\[
\begin{cases}
\mathcal{F} \simeq pr_{2*}(pr_{1}^* \mathcal{F} \otimes \mathcal{K}(\mathcal{E}, s)) \\
\mathcal{F} \simeq pr_{1*}(pr_{2}^* \mathcal{F} \otimes \mathcal{K}(\mathcal{E}, s)).
\end{cases}
\]
where $\mathcal{E}=\Omega^1(1) \boxtimes \mathcal{O}(-1)$ endowed with the cosection defined at the beginning of \S \ref{3.3}. These isomorphisms are the crucial tool to study coherent sheaves on $\mathbb{P}(W)$ ({see} \cite{Beilinson}). The first one is closely related to the Syzygy theorem, but we will be more interested in the second one. The corresponding Beilinson spectral sequence satisfies
\begin{equation} \label{ssb}
\begin{cases}
E_1^{-p, q} = \Omega^p(p) \otimes_{\mathbb{C}} H^q(\mathbb{P}(W), \mathcal{F}(-p)) \\
E_{\infty}^{-p, p} =  Gr^{-p} (\mathcal{F})
\end{cases}
\end{equation}
where we filter by columns, so that $\{0\} = F^1 \mathcal{F} \subseteq F^0 \mathcal{F} \subseteq F^{-1} \mathcal{F} \subseteq \ldots \subseteq F^{-N} \mathcal{F} =\mathcal{F}$. Assume now that $\mathcal{F}=j_{X/\mathbb{P}(W)*} \mathcal{O}_X$ where $X$ is a smooth subvariety of $\mathbb{P}(W)$ of dimension $d$, and write $j=j_{X/\mathbb{P}(W)}$ for simplicity. Then \eqref{ssb} becomes

\[
\begin{cases}
E_1^{-p, q} = \Omega^p(p) \otimes_{\mathbb{C}} H^q(X, \mathcal{O}_X(-p)) \\
E_{\infty}^{-p, p} =  Gr^{p} (j_*\mathcal{O}_X).
\end{cases}
\]
Thanks to Kodaira--Nakano vanishing  theorem, $E_1^{-p,q}=0$ if $p>0$ and $q \neq d$. Hence the only nonzero arrows of the spectral sequence appear in the picture below, where we recall that $c=N-d$ is the codimension of $X$ in $\mathbb{P}(W)$:

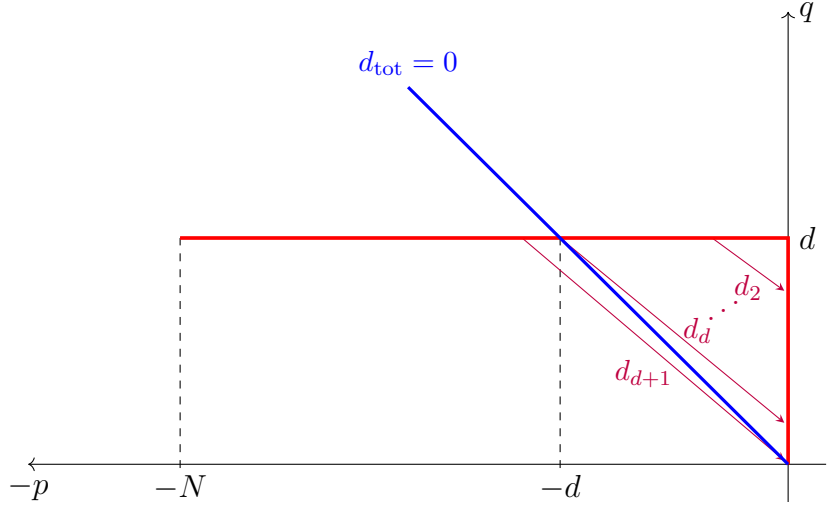
\[\begin{tikzpicture}
	\draw[->] (0.5,0) -- (-10,0);
	\draw (-10,0) node[below] {$-p$};
	\draw (-8,0) node[below] {$-N$};
	\draw (-3,0) node[below] {$-d$};
				\draw [->] (0,-0.5) -- (0,6);
	\draw (0,6) node[right] {$q$};
		\draw (0,3) node[right] {$d$};
		\draw[dashed] (-8,-0.05) -- (-8,3);
		\draw[dashed] (-3,-0.05) -- (-3,3);
		\draw[purple,-stealth] (-1,3) to node[below]{\small$d_2$}(-0.05,2.3);
				\draw[purple,-stealth] (-3,3) to node[right]{\small$d_{d}$}(-0.05,0.55);
		\draw[purple,-stealth] (-3.5,3) to node[below]{\small$d_{d+1}$~~~}(-0.05,0.05);
		\draw (-.85,2.15) node[purple] {\reflectbox{$\ddots$}} ;
\draw[red,very thick] (0,0)--(0,3) -- (-8,3);
\draw[red,very thick] (0,0)--(0,3) -- (-8,3);
\draw[blue,very thick] (0,0)--(-5,5) node[above]
{\small$d_\mathrm{tot}=0$};

\end{tikzpicture}
\]
It follows that the only possibly nonzero $\mathrm{E}_{\infty}$ components are $\mathrm{E}_{\infty}^{-d,d}$ and $\mathrm{E}_{\infty}^{0,0}$, which means that the filtration on $j_* \mathcal{O}_X$ is of the type
\[
F^0 (j_* \mathcal{O}_X) = \ldots = F^{-(d-1)} (j_* \mathcal{O}_X) \subseteq F^{-d} (j_* \mathcal{O}_X) = \ldots = F^{-N} (j_* \mathcal{O}_X) = j_* \mathcal{O}_X.
\]

\begin{lemma}
$E_{\infty}^{-d,d}=\{0\}$.
\end{lemma}

\begin{proof}
We use the fonctoriality of the Beilinson spectral sequence. For $X=\mathbb P(W)$, $d=N$ so the filtration is trivial. Hence we have a commutative diagram
\[
\xymatrix{
\mathcal{O}_{\mathbb P(W)} \ar[r] & j_* \mathcal{O}_X \\
F^{0} \mathcal{O}_{\mathbb P(W)} \ar[r] \ar[u] &   \ar[u] F^{0} (j_* \mathcal{O}_X)
}
\]
Since the left vertical arrow is an isomorphism, the right one is surjective, so it is an isomorphism.
\end{proof}

\begin{lemma}\label{jaune}
$\ker d_1^{-(d+1),d}$ is locally free.
\end{lemma}

\begin{proof}
We have $E_{\infty}^{-d,d}=\{0\}$. Hence we have a bunch of isomorphisms and exact sequences, where locally free sheaves are highlighted in yellow:
\begin{align*}
& 0 \rightarrow \ker d_1^{-1,d} \rightarrow \highlight{E_1^{-1,d}} \xrightarrow{d_1^{-1, d}} \highlight{E_1^{0,d}} \rightarrow 0 \\
& \ker d_1^{-1,d} \simeq \mathrm{Im}\, d_1^{-2, d} \\
& 0 \rightarrow \ker d_1^{-2,d} \rightarrow \highlight{E_1^{-2,d}} \xrightarrow{d_1^{-2, d}} \mathrm{Im} \,d_1^{-2,d} \rightarrow 0 \\
& 0 \rightarrow \mathrm{Im} \, d_1^{-3,d} \rightarrow \ker d_1^{-2,d} \rightarrow E_2^{-2,d}\overset{d_2}{\simeq} E_2^{0,d-1}\simeq \highlight{E_1^{0,d-1}} \rightarrow 0 \\
& 0 \rightarrow \ker d_1^{-3,d} \rightarrow \highlight{E_1^{-3,d}} \xrightarrow{d_1^{-3, d}} \mathrm{Im} \, d_1^{-3,d} \rightarrow 0 \\
& 0 \rightarrow \mathrm{Im} \,d_1^{-4,d} \rightarrow \ker d_1^{-3,d} \rightarrow E_2^{-3,d}\simeq E_3^{-3,d}\overset{d_3}{\simeq} E_3^{0,d-2}\simeq \highlight{E_1^{0,d-2}} \rightarrow 0 \\
& \vdots \\
& 0 \rightarrow \ker d_d^{-d, d} \rightarrow \highlight{E_1^{-d,d}} \xrightarrow{d_1^{-d,d}} \mathrm{Im}\, d_1^{-d,d} \rightarrow 0 \\
&\ker d_d^{-d, d} \simeq   \mathrm{Im} \, d_1^{-(d+1),d}  \rightarrow 0 \\
&0 \rightarrow \ker d_1^{-(d+1),d} \rightarrow \highlight{E_1^{-{d+1},d}} \xrightarrow{d_1^{-(d+1), d}} \mathrm{Im} \, d_1^{-(d+1),d} \rightarrow 0. \quad \numberthis \label{alpha}
\end{align*}
The result follows.
%
\end{proof}

Let us consider the complex $\tau^{<-d} E_1^{\bullet, q}$, that we realize as
\[
E_1^{-N, d} \rightarrow E_1^{-(N-1), d} \rightarrow \ldots \rightarrow E_1^{-(d+2), d} \rightarrow \ker d_1^{-(d+1), d}.
\]
We can form the composition
\[
\varphi \colon \ker d_1^{-(d+1), d} \rightarrow E_2^{-(d+1), d} \simeq E_{d+1}^{-(d+1), d} \xhookrightarrow{d_{d+1}^{-(d+1), d}} E_{d+1}^{0,0} \simeq E_1^{0,0}=\mathcal{O}
\]
where we put $\mathcal{O}$ in degree $0$.
\begin{proposition} \label{resolution}
The complex 
\[
E_1^{-N, d} \rightarrow E_1^{-(N-1), d} \rightarrow \ldots \rightarrow E_1^{-(d+2), d} \rightarrow \ker d_1^{-(d+1), d } \xrightarrow{\varphi} \mathcal{O}
\]
is a locally free resolution of $j_*\mathcal{O}_X$. 
\end{proposition}

\begin{proof}
First note that it is indeed a complex: since $\varphi$ factors through $E_2^{-(d+1), d}$, it vanishes on the image of $d_1^{-(d+2), d}$. The complex $E_1^{\bullet, q}$ has no cohomology in degrees $\llbracket -N, -(d+2) \rrbracket$. Since $\ker \varphi = \mathrm{im}\, d_1^{-(d+2), d}$, there is no cohomology neither in degree $-1$. Hence we see that the complex under investigation is a  resolution of $E_{\infty}^{0,0}$, which is $j_*(\mathcal{O}_X)$. Lastly, we proved in the previous \Cref{jaune} that $\ker d_1^{-(d+1),d}$ is locally free, thus the resolution is locally free.
\end{proof}

\subsection{The canonical complex (II)} \label{vegas}
In this section, we prove the following explicit result computing the derived self-intersection of $X$ in $\mathbb{P}(W)$.

\begin{theorem} \label{AS} We have
$\mathcal{O}_X \overset{\mathbb{L}}{\otimes}_{\mathbb P(W)} \mathcal{O}_X \simeq (\tau^{<d} \mathscr{C}(X))[-d] \oplus \mathcal{O}_X$. 
\end{theorem}

We provide two different proofs of this result, one relying 
on \Cref{brique}, and the other hinging on our study of the Beilinson spectral sequence in the previous section. Both proofs are very close since they rely on the Koszul complex on $\mathbb{P}(W) \times \mathbb{P}(W)$ defining the diagonal.

\begin{proof}[Proof 1]
Thanks to \Cref{resolution}, we can solve $j_* \mathcal{O}_X$ by the complex of locally free sheaves
\[
E_1^{-N, d} \rightarrow E_1^{-(N-1), d} \rightarrow \ldots \rightarrow E_1^{-(d+2), d} \rightarrow \ker d_1^{-(d+1), d } \xrightarrow{\varphi} \mathcal{O}
\]
on $\mathbb{P}(W)$. Hence the derived intersection we are looking for is equal to the restriction of this complex to $X$. Since $\mathrm{Im} \varphi$ is the ideal sheaf of $X$, $\varphi_{|X}=0$, which gives the complex
\[
{E_1^{-N, d}}_{|X} \rightarrow {E_1^{-(N-1), d}}_{|X} \rightarrow \ldots \rightarrow {E_1^{-(d+2), d}}_{|X} \rightarrow ({\ker d_1^{-(d+1), d }})_{|X} \xrightarrow{0} \mathcal{O}_X.
\]
It remains to prove that the natural morphism
\[
({\ker d_1^{-(d+1), d }})_{|X} \rightarrow \ker {d_1^{-(d+1), d}}_{|X}
\]
is an isomorphism. This follows from restristing the exact sequence \eqref{alpha} to $X$.
\end{proof}

\begin{proof}[Proof 2] Using \Cref{brique}, we see that 
\[
\mathcal{O}_X \overset{\mathbb{L}}{\otimes}_{\mathbb P(W)} \mathcal{O}_X \simeq \mathrm{pr}_{1*} \mathcal K(\mathcal E|_{X^2},s|_{X^2})
\] 
in $\mathrm{D}(\mathbb{P}(W))$.
We have a Beilinson-type spectral sequence such that
\begin{align*}
E_1^{-p,q}&=R^q\mathrm{pr}_{1*}(\Omega^p_{|X}(p)\boxtimes\mathcal O(-p))\\
&=\Omega^p_{|X}(p)\otimes_\mathbb C H^q(X,\mathcal O(-p)).
\end{align*}
By Kodaira--Nakano vanishing theorem, the term $H^q(X,\mathcal O(-p))$ is nonzero only if $p=0$ or $q=d$. To visualize the situation, we take a double complex giving an injective resolution of $\mathcal{K}(\mathcal E|_{X^2},s|_{X^2})$, and apply $pr_{1*}$. This gives a double complex $K$, which is quasi-isomorphic to $\mathcal{O}_X {\otimes}^{\mathbb{L}}_{\mathbb P(W)} \mathcal{O}_X$. Since
\[
\mathcal{O}_X \overset{\mathbb{L}}{\otimes}_{\mathbb P(W)} \mathcal{O}_X \simeq \tau^{<0} (\mathcal{O}_X \overset{\mathbb{L}}{\otimes}_{\mathbb P(W)} \mathcal{O}_X) \oplus j_*\mathcal{O}_X, 
\]
the complex $K$ satisfies the following properties:
\begin{enumerate}
\item[--] $\mathrm{Tot}(K)$ is concentrated in degree $\llbracket-c,0\rrbracket$;
\item[--] $K\simeq \mathcal O_X\oplus \tau_\mathrm{tot}^{<0}(K)$.
\end{enumerate}
This can be depicted as follows : the complex is cohomologically concentrated in the blue zone, and the $E_1$ page of the spectral sequence
lies on the red line.
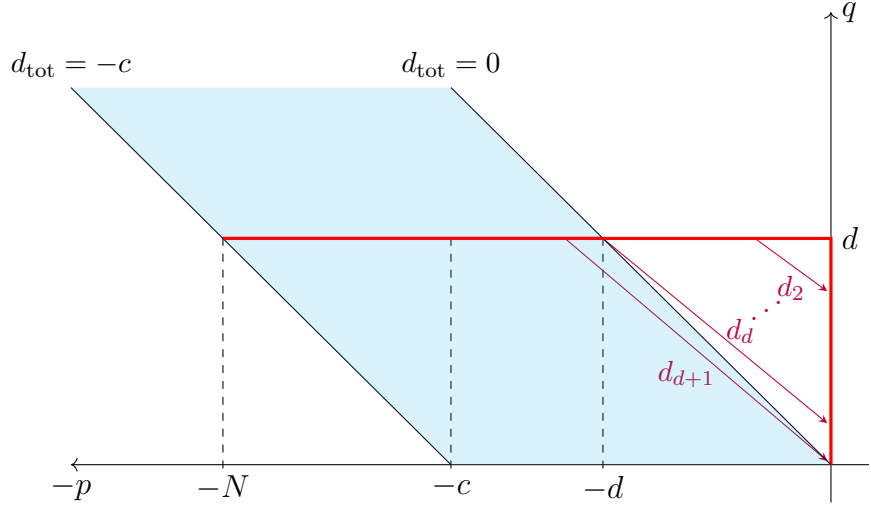
\[\begin{tikzpicture}
	\draw[->] (0.5,0) -- (-10,0);
	\draw (-10,0) node[below] {$-p$};
	\draw (-8,0) node[below] {$-N$};
	\draw (-5,0) node[below] {$-c$};
	\draw (-3,0) node[below] {$-d$};
				\draw [->] (0,-0.5) -- (0,6);
	\draw (0,6) node[right] {$q$};
		\draw (0,3) node[right] {$d$};
		\draw[dashed] (-8,-0.05) -- (-8,3);
		\draw[dashed] (-5,-0.05) -- (-5,3);
		\draw[dashed] (-3,-0.05) -- (-3,3);
		\draw[purple,-stealth] (-1,3) to node[below]{\small$d_2$}(-0.05,2.3);
				\draw[purple,-stealth] (-3,3) to node[right]{\small$d_{d}$}(-0.05,0.55);
		\draw[purple,-stealth] (-3.5,3) to node[below]{\small$d_{d+1}$~~~}(-0.05,0.05);
		\draw (-.85,2.15) node[purple] {\reflectbox{$\ddots$}} ;
								\draw[red,very thick] (0,0)--(0,3) -- (-8,3);
			\draw(0,0)--(-5,5) node[above]{\small$d_\mathrm{tot}=0$};
			\draw(-5,0)--(-10,5) node[above]{\small$d_\mathrm{tot}=-c$};
				\clip (0.5,-0.5) -- (-4.5,-0.5) -- (-10,5) -- (-5,5);
	\fill[cyan,opacity=.15] (0.5,0) rectangle (-10,5);
							\draw[red,very thick] (0,0)--(0,3) -- (-8,3);
\end{tikzpicture}\]
For any $q$ we have a triangle associated to the filtration by the rows\[
	\tau_\mathrm{row}^{>q}K\to	\tau_\mathrm{row}^{\ge q}K\to E_1^{\bullet,q}\overset{+1}{\to}
	\]
	which yields\[
	\tau_\mathrm{tot}^{<0}\tau_\mathrm{row}^{>q}K\to	\tau_\mathrm{tot}^{<0}\tau_\mathrm{row}^{\ge q}K\to \tau_\mathrm{tot}^{<0}E_1^{\bullet,q}\overset{+1}{\to}
	\]
where the cofiber is $0$ whenever $q\neq d$. The result follows by induction on $q$, since the complex $E_1^{\bullet, q}$ is exactly the complex $\mathscr{C}(X)$.
\end{proof}

%

\end{document}